\documentclass[12pt]{amsart}
\usepackage{amsmath}
\usepackage{footnote}
\usepackage{amssymb}
\usepackage[dvipspdf]{xcolor}

\def\zhou#1 {\fbox {\footnote {\ }}\ \footnotetext { From Zhou: {\color{red}#1}}}

\def\alex#1 {\fbox {\footnote {\ }}\ \footnotetext { From Alex: {\color{blue}#1}}}

\def\qi#1 {\fbox {\footnote {\ }}\ \footnotetext { From Qi: {\color{red}#1}}}

\newcommand{\bbZ}{{\mathbb Z}}
\newcommand{\tr}{{\rm Tr}}
\newcommand{\gf}{{\mathbb F}}
\newcommand{\bC}{{\mathbb C}}
\newcommand{\bQ}{{\mathbb Q}}

\newcommand{\aut}{{\rm Aut}}

\newtheorem{theorem}{Theorem}[section]
\newtheorem{lemma}[theorem]{Lemma}
\newtheorem{proposition}[theorem]{Proposition}
\newtheorem{corollary}[theorem]{Corollary}

\newtheorem{fact}{Fact}

\theoremstyle{definition}
\newtheorem{definition}[theorem]{Definition}

\theoremstyle{remark}
\newtheorem{remark}[theorem]{Remark}

\numberwithin{equation}{section}

\begin{document}

\title[Difference balanced functions]{Difference balanced functions and their generalized difference sets}

\author{Alexander Pott}
\address{Institute of Algebra and Geometry, Faculty of Mathematics, Otto-von-Guericke University Magdeburg, Universit\"atsplatz 2, 39106, Magdeburg, Germany}
\curraddr{}
\email{alexander.pott@ovgu.de}
\thanks{}

\author{Qi Wang} 
\address{Institute of Algebra and Geometry, Faculty of Mathematics, Otto-von-Guericke University Magdeburg, Universit\"atsplatz 2, 39106, Magdeburg, Germany}
\curraddr{Department of Computer Science and Engineering, The Hong Kong University of Science and Technology, Clear Water Bay, Kowloon, Hong Kong}
\email{qi.wang@ovgu.de}
\thanks{}

\subjclass[2010]{Primary 05B10; Secondary 05B30}
\keywords{difference balanced function, $d$-homogenerous function, generalized difference set, $p$-ary sequence with ideal autocorrelation, two-tuple balanced function}

\date{}

\dedicatory{}

\begin{abstract}
  Difference balanced functions from $\gf_{q^n}^*$ to $\gf_q$ are closely related to combinatorial designs and naturally define $p$-ary sequences with the ideal two-level autocorrelation. In the literature, all existing such functions are associated with the $d$-homogeneous property, and it was conjectured by Gong and Song that difference balanced functions must be $d$-homogeneous. First we characterize difference balanced functions by generalized difference sets with respect to two exceptional subgroups. We then derive several necessary and sufficient conditions for $d$-homogeneous difference balanced functions. In particular, we reveal an unexpected equivalence between the $d$-homogeneous property and multipliers of generalized difference sets. By determining these multipliers, we prove the Gong-Song conjecture for $q$ prime. Furthermore, we show that every difference balanced function must be balanced or an affine shift of a balanced function.
  

    
\end{abstract}

\maketitle

\section{Introduction}\label{sec-intro}

The interaction of the additive and multiplicative structure of finite fields plays an important role in constructing interesting functions and sequences in coding theory, cryptography, communications, etc. Difference balanced functions are such an example. They can generate $p$-ary sequences with the ideal two-level autocorrelation. Throughout this paper, let $q = p^m$ with $p$ an {\em odd} prime. A function $f: \gf_{q^n}^* \rightarrow \gf_q$ is said to be {\em balanced}, if 
$$
\left| \{ x \in \gf_{q^n}^* : f(x) = 0 \} \right| = q^{n-1} - 1,
$$
and 
$$
\left| \{ x \in \gf_{q^n}^* : f(x) = b \} \right| = q^{n-1} ,
$$
for each $b \in \gf_q^*$. 
A function $g: \gf_{q^n}^* \rightarrow \gf_q$ is called {\em difference balanced}, if $f(ax) - f(x)$ is balanced for each $a \in \gf_{q^n}^*\setminus \{1\}$. 
Since $\bbZ_{q^n - 1}$ is isomorphic to $\gf_{q^n}^*$, if $q = p$, we may naturally define a $p$-ary sequence ${\bf s}$ of period $p^n-1$ using a function $f: \gf_{p^n}^* \rightarrow \gf_p$ by $s_i := f(\theta^i)$, where $\theta$ is a primitive element of $\gf_{p^n}$. It was proved in~\cite{No04} that the function $f:
\gf_{p^n}^* \rightarrow \gf_p$ is difference balanced if and only if the sequence ${\bf s}$ has the ideal two-level autocorrelation as
$$
{\mathcal C}_{\bf s}(\tau) = \left\{\begin{array}{ll}
  p^n - 1 & \textrm{ if $\tau \equiv 0 \pmod{p^n - 1}$,} \\
  -1  & \textrm{ if $\tau \not\equiv 0 \pmod{p^n - 1}$,}
\end{array}\right.
$$
where
$$
{\mathcal C}_{\bf s}(\tau) = \sum_{0 \leq i < p^n - 1} \zeta_p^{s_{i+\tau} - s_i},
$$
and $\zeta_p$ is a $p$-th root of unity. In the literature, there are only a few constructions of difference balanced functions from $\gf_{q^n}^*$ to $\gf_q$:
\begin{itemize}
  \item[--] the trace function $ f(x) = \tr_{q^n/q}(x)$, where $\tr_{q^n/q}$ denotes the trace function from $\gf_{q^n}$ to $\gf_q$;
  \item[--] the Helleseth-Gong function~\cite{HKM01,HG02}: 
   $$
   f(x) = \tr_{q^n/q} \left( \sum_{i=0}^{\ell} u_i x^{(q^{2ki} + 1)/2} \right),
   $$
   where $n = (2\ell + 1)k$,  $b_i$'s are defined as $b_0 = 1$, $b_{ij} = (-1)^i$ for $1 \leq j \leq 2\ell$ an integer with $\gcd(j, 2\ell+1) = 1$, and $b_i  = b_{2\ell + 1 - i}$ for $i = 1, \ldots, \ell$, and $u_0 = b_0 /2 = (p+1)/2$, $u_i = b_{2i}$ for $i = 1, \ldots, \ell$ (all the indices of the $b_i$'s are taken modulo $2\ell + 1$); 
  \item[--] the Lin function~\cite{Lin98,Arasu10,Hu13} in characteristic $3$: 
    $$
    f(x) = \tr_{3^n/3} (x + x^e),
    $$
    where $n \geq 3$ is an odd integer and $e = 2\cdot 3^{(n-1)/2} + 1$;
  \item[--] the cascaded composition of the functions above using the Gordon-Mills-Welch method~\cite{GMW62} or the method by No~\cite{No04}.
\end{itemize}
A function $f: \gf_{q^n}^* \rightarrow \gf_q$ is called {\em $d$-homogeneous}, if there exists an integer $d$ with $\gcd(d, q - 1) = 1$ such that $f(a x) = a^d f(x)$ for all $x \in \gf_{q^n}^*$ and each $a \in \gf_q^*$. It is observed that all currently known difference balanced functions listed above are $d$-homogeneous. Using this observation, these difference balanced functions can be used to obtain cyclic difference sets or relative difference sets with Singer parameters~\cite{CX03,CX032,No04,NSH04,KNCH05} (for a recent survey, we refer to~\cite{Xiang05}). The two-tuple balance property (see Section~\ref{sec-twotuple} for the definition) of difference balanced functions was derived in~\cite{GG05} (see also~\cite{GS06}), and it was further conjectured by Gong and Song~\cite{GS06} that difference balanced functions from $\gf_{q^n}^*$ to $\gf_q$ must be $d$-homogeneous (an equivalent formulation in terms of the {\em cyclic array structure} was given in~\cite{GS06}). 


In this paper, we characterize difference balanced functions by certain generalized difference sets with two exceptional subgroups defined in~\cite{PWZ12} (see Section~\ref{sec-gds}), and then derive several necessary and sufficient conditions for $d$-homogeneous difference balanced functions from $\gf_{q^n}^*$ to $\gf_q$. Among these equivalent conditions, we observe that the $d$-homogeneous property of functions is equivalent to certain multipliers of
the corresponding generalized difference sets. By determining these multipliers explicitly, we are able to prove the conjecture by Gong and Song for $q$ prime. Moreover, we also show that every difference balanced function from $\gf_{q^n}^*$ to $\gf_q$ must be balanced or an affine shift of a balanced function.

The rest of the present paper is organized as follows. In Section~\ref{sec-pre}, generalized difference sets and some other objects related to difference balanced functions are introduced. In Section~\ref{sec-balance}, the balance property of difference balanced functions from $\gf_{q^n}^*$ to $\gf_q$ is derived in general. The links between $d$-homogeneous difference balanced functions and other combinatorial objects are revealed by giving necessary and sufficient conditions in Section~\ref{sec-equi}. In Section~\ref{sec-multiplier}, the Gong-Song conjecture is then proved for $q$ prime using certain numerical multipliers of the corresponding generalized difference sets. Finally, we conclude the paper with some open problems. 


\section{Preliminaries}\label{sec-pre}

In this section, we first introduce difference sets, relative difference sets, and generalized difference sets. Difference balanced functions from $\gf_{q^n}^*$ to $\gf_q$ can then be characterized by generalized difference sets with respect to two exceptional subgroups of orders $q$ and $q^n-1$. We also give the definition of two-tuple balanced functions, which are in fact equivalent to $d$-homogeneous difference balanced functions (we will explain this relation in Section~\ref{sec-equi}).

\subsection{Generalized difference sets}\label{sec-gds}

Let $G$ be an abelian group of order $v$ (written multiplicatively). Let $N_1, \ldots, N_r$ be subgroups of $G$ of order $n_1, \ldots, n_r$ and intersect pairwise trivially. A $k$-subset $D$ of $G$ is a $(v; n_1, \ldots, n_r; k, \lambda; \lambda_1, \ldots, \lambda_r )$ {\em generalized difference set} relative to the subgroups $N_i$, if the list of differences $d_1 d_2^{-1}$ with $d_1, d_2 \in D$ and $d_1 \ne d_2$, contains each element in $G \setminus (N_1 \cup N_2 \cup \cdots
N_r)$ exactly $\lambda$ times, and every non-identity element in $N_i$ exactly $\lambda_i$ times. The $N_i$'s are called {\em exceptional} subgroups. A generalized difference set $D$ is called {\em cyclic} (or {\em abelian}) if the group $G$ is cyclic (or abelian). We may identify a subset $D$ of a group $G$ with the group ring element $\sum_{d \in D} d$, which is also denoted by $D$ (by abuse of notation). Addition and multiplication in group rings are defined as:
$$
\sum_{g \in G} a_g g + \sum_{g \in G} b_g g =  \sum_{g \in G} (a_g + b_g) g,
$$
and
$$
\sum_{g \in G} a_g g  \sum_{g \in G} b_g g = \sum_{g \in G} \left( \sum_{h \in G} a_h b_{h^{-1}g} \right) g .
$$
We also define $D^{(t)} := \sum_{g \in G} d_g g^t$, for $D =
\sum_{g \in G} d_g g \in \bbZ[G]$ and $t$ an integer. In the language of group rings, in $\bbZ[G]$ we have the following characterization of (generalized) difference sets. 

\begin{proposition}\label{pro-gds}\cite{PWZ12}
The set $D$ is a generalized $(v; n_1, \ldots, n_r; k, \lambda; \lambda_1, \ldots, \lambda_r)$ difference set relative to the subgroups $N_i$ if and only if 
\begin{eqnarray*}
  D D^{(-1)} & = &  k  - (\lambda(1-r) + \lambda_1 + \cdots + \lambda_r) + \\
  & &  \quad +\  \lambda (G - N_1 - \cdots - N_r) + \lambda_1 N_1 + \cdots + \lambda_r N_r.
\end{eqnarray*}
\end{proposition}

The $(v,k, \lambda)$ {\em difference sets} are special types of generalized difference sets with parameters $(v; 1; k, \lambda; 0)$, the $(m,n,k, \lambda)$ {\em relative difference sets} are those with parameters $(mn; n; k, \lambda; 0)$, and the $(m,n,k,\lambda_1, \lambda_2)$ {\em divisible difference sets} are those with parameters $(mn; n; k, \lambda_2; \lambda_1)$. By convention, we still use $(v, k, \lambda)$, $(m,n,k,\lambda)$ and $(m,n,k,\lambda_1,\lambda_2)$ to denote the parameters of difference sets, relative difference sets, and divisible difference sets, respectively. For more information on these combinatorial objects, we refer to~\cite{BJL99,Lan83,Pott95,Schmidt02}.

Let $H \leq G$ be a subgroup of $G$, and let $\rho_H$ denote the canonical epimorphism $G \rightarrow G/H$ with $\rho_H (g) = gH$. For a generalized difference set $D$, if $H \leq N_i$ for some $N_i$ with $\lambda_i = 0$, then $\rho_H(D)$ can be interpreted as a subset of $G/H$. In particular, we have the following result on projections of generalized difference sets with respect to two exceptional subgroups.

\begin{proposition}\label{pro-project}
Let $D$ be a generalized $(v; n_1, n_2; k, \lambda; 0, 0)$ difference set in $G$ relative to $N_1$ and $N_2$. If $H \leq N_1$ is a subgroup of order $m$, then $\rho_H(D)$ is a generalized $(v/m; n_1/m, n_2; k, m \lambda; 0, \lambda (m-1))$ difference set.
\end{proposition}

Generalized difference sets with parameters $(n(n-1); n, n-1; n-1,1; 0, 0)$ in a group $G$ relative to subgroups $N_1$ of order $n$ and $N_2$ of order $n-1$ are called {\em direct product difference sets}~\cite{Ganley77,Pott95} (for their connections to projective planes, see also~\cite{PWZ12}). By Proposition~\ref{pro-project}, the trace function can be viewed as the projection of the classical direct product difference set: $\{(x,x): x \in
\gf_{q^n}^*\} \subseteq (\gf_{q^n}^*, \cdot) \times (\gf_{q^n}, +)$ (see~\cite{PWZ12}). In our language of generalized difference sets, the set $\{(x, \tr_{q^n/q}(x)) : x \in \gf_{q^n}^*\} \subseteq (\gf_{q^n}^*, \cdot ) \times (\gf_q, +)$ has parameters $(q(q^n-1); q, q^n - 1; q^n -1, q^{n-1}; 0, q^{n-1} - 1)$. In fact, every difference balanced function can be characterized as a generalized difference set with such parameters: a function $f: \gf_{q^n}^* \rightarrow \gf_q$ is
difference balanced, if and only if the set $D := \{(x, f(x)): x \in \gf_{q^n}^* \} \subseteq G = N_2 \times N_1$ is a generalized $(q(q^n-1); q, q^n -1; q^n -1, q^{n-1}; 0, q^{n-1} -1 )$ difference set relative to $N_1 = (\gf_q,+)$ and $N_2 = (\gf_{q^n}^*,\cdot)$. In group ring language, 
\begin{equation}\label{eqn-gds}
  D D^{(-1)} = q^n  + q^{n-1} G - q^{n-1} N_1 - N_2.
\end{equation}


Character theory is a powerful tool to investigate difference sets. Let $G$ be an abelian group. A character $\chi$ of $G$ is a homomorphism $\chi: G \rightarrow \bC^*$. The set of all such characters forms a group $\hat{G}$ isomorphic to $G$, and the identity element of $\hat{G}$ (which maps every element in $G$ to $1$) is called the {\em principal} character, denoted by $\chi_0$. The following property of characters is well-known.

\begin{fact}[Orthogonality relations]\label{fact-charorth}
Let $G$ be an abelian group of order $v$ with identity $e$. Then
$$
\sum_{\chi \in \hat{G}} \chi(g) = \left\{ \begin{array}{cl}
  0 & \textrm{ if $g \ne e$,}\\
  v & \textrm{ if $g = e$,}
\end{array}\right.
$$
and
$$
\sum_{g \in G} \chi(g) = \left\{ \begin{array}{cl}
  0 & \textrm{ if $\chi \ne \chi_0$,}\\
  v & \textrm{ if $\chi = \chi_0$.}
\end{array}\right.
$$
\end{fact}

If $\chi$ is a character of an abelian group, then $\chi(D^{(-1)}) = \overline{\chi(D)}$, i.e., $\chi(D^{(-1)})$ is the complex conjugate of $\chi(D)$. By Fact~\ref{fact-charorth} and (\ref{eqn-gds}), the character values of $D D^{(-1)}$ are
\begin{equation}\label{eqn-char}
  \chi(D)\overline{\chi(D)} = \left\{\begin{array}{cl} 
    (q^n-1)^2 & \textrm{ if $\chi = \chi_0$,}\\
    0 & \textrm{ if $\chi|N_1 = \chi_0$ and $\chi \ne \chi_0$,}\\
    1 & \textrm{ if $\chi|N_2 = \chi_0$ and $\chi \ne \chi_0$,}\\
    q^n & \textrm{ if $\chi|N_1 \ne \chi_0$ and $\chi|N_2 \ne \chi_0$.}
  \end{array}\right.
\end{equation}

For an abelian (generalized) difference set $D$, multipliers are useful especially when we want to obtain information about $\chi(D)$ from $\chi(D) \overline{\chi(D)}$. {\em Multipliers} are group homomorphisms $\psi$ such that $\psi(D) := \{ \psi(x): x \in D \}$ is a translate $gD$ of $D$, where $gD = \{g \cdot d : d \in D\}$. If an integer $t$ relatively prime to $|G|$ has the property that $D^{(t)} = gD$ for some $g \in G$, then $t$ is called a {\em numerical
multiplier} of $D$. It is convenient to use the mixture of additive and multiplicative notation since in the following the group $N_1$ is the additive group and $N_2$ is the multiplicative group of finite fields. In particular, for a generalized difference set $D = \{ (x , y) \}  \subseteq (N_2, \cdot) \times (N_1, +)$, by $D^{(t_1,t_2)}$ we denote the set $\{ (x^{t_1},  y\cdot t_2): (x,y) \in D\}$, and by $ (h_1,h_2) D$ we denote a translate of $D$: $\{( h_1  x, y + h_2): (x,y) \in D \}$.

\subsection{Two-tuple balanced functions}\label{sec-twotuple}

The two-tuple balance property is one of some randomness measurements for sequences, and it is also important for determining the autocorrelation of GMW sequences~\cite{GG05}. We now describe the two-tuple balance property of functions from $\gf_{q^n}^*$ to $\gf_q$.   

\begin{definition}\label{def-twotuple}\cite{GG05}
  For a function $f: \gf_{q^n}^* \rightarrow \gf_q$, define
  \begin{equation*}
  N_{b_1,b_2} (a) = | \{ x \in \gf_{q^n}^*: (f(x), f(a x)) = (b_1, b_2) \} |,
\end{equation*}
  where $a \ne 0$ and $b_1, b_2 \in \gf_q$. Then $f$ is called {\em two-tuple balanced} if the following conditions are satisfied:
  \begin{itemize}
    \item[(i)] If $a \not\in \gf_q$, then
      $$
      N_{b_1,b_2} (a) = \left\{\begin{array}{cl}
	q^{n-2} - 1 & \textrm{ if $(b_1, b_2) = (0, 0)$,}\\
	q^{n-2} & \textrm{ if $(b_1, b_2) \ne (0, 0)$.}
      \end{array}\right.
      $$
    \item[(ii)] If $a \in \gf_q^*$, then there exists some $\mu \in \gf_q$ such that $(f(x), f(a x)) = (b_1, \mu b_1)$ for all $x \in \gf_{q^n}^*$ and 
      $$
      N_{b_1,b_2} (a) = \left\{\begin{array}{cl}
	q^{n-1} - 1 & \textrm{ if $(b_1, b_2) = (0, 0)$,}\\
	q^{n-1} & \textrm{ if $(b_1, b_2) = (b_1, \mu b_1)$ for $b_1 \in \gf_q$,} \\
	0 & \textrm{ otherwise.}
      \end{array}\right.
      $$
  \end{itemize}
\end{definition}

In~\cite{GS06}, the following relation between the two-tuple balance property and the difference balance property was given. 

\begin{theorem}~\cite{GS06}\label{thm-ttb-db}
  If a difference balanced function $f: \gf_{q^n}^* \rightarrow \gf_q$ is $d$-homogeneous, then $f$ must be two-tuple balanced.
\end{theorem}

Conversely, for a function from $\gf_{q^n}^*$ to $\gf_q$, the two-tuple balance property implies both the difference balance property and the balance property (see~\cite{GS06}). In Section~\ref{sec-equi}, we will see that the converse of Theorem~\ref{thm-ttb-db} also holds, i.e., the two-tuple balance property implies the $d$-homogeneity.

\section{The balance property of difference balanced functions}\label{sec-balance}

It was proved in~\cite{LG00} that every difference balanced function $f$ from $\gf_{q^n}^*$ to $\gf_q$ must be balanced or an affine shift of a balanced function $g$, i.e., $f = g + b$ for some $b \in \gf_q$, if $q$ is a prime. Now we show that this is also true for $q$ a prime power.    


\begin{theorem}\label{thm-balance}
  Every difference balanced function $f: \gf_{q^n}^* \rightarrow \gf_q$ must be of the form $f = g + b$ for some $b \in \gf_q$, where the function $g$ is balanced. 
\end{theorem}

\begin{proof}
  Define $D_b := \{(x,b) \in G : f(x) = b \}$ and $d_b = |D_b|$, where $G = \gf_{q^n}^* \times \gf_q$ and $b \in \gf_q$. We count the number of certain elements in $DD^{(-1)}$ in two ways: one by direct calculation, and the other by the group ring equation (\ref{eqn-gds}).

  First, by counting the number of $(a,0)$ in $DD^{(-1)}$ where $a \in \gf_{q^n}^*$, we have
  \begin{eqnarray}\label{eqn-ban1st}
    \sum_{b \in \gf_q} d_b^2 & =  & q^n + q^{n-1} (q^n - 1) - q^{n-1} - (q^n - 1)  \nonumber \\
    & = & q^{2n-1} - 2 q^{n-1} + 1.
  \end{eqnarray}
  Second, we calculate the number of $(a,b)$ in $DD^{(-1)}$ where $a \in \gf_{q^n}^*$ and $b \ne 0$:
  \begin{eqnarray}\label{eqn-ban2nd}
    \sum_{\substack{b_1 \ne b_2 \\ b_1, b_2 \in \gf_q}} 2 d_{b_1} d_{b_2} & = & q^{n-1} (q^n-1) (q-1) - q^{n-1}(q-1) \nonumber \\
     & = & q^{2n} - 2 q^n - q^{2n-1} + 2 q^{n-1} .
  \end{eqnarray}
  With (\ref{eqn-ban1st}) and (\ref{eqn-ban2nd}), we get 
  \begin{equation}\label{eqn-ban3rd}
  (q-1) \sum_{b \in \gf_q} d_b^2 - \sum_{\substack{b_1 \ne b_2 \\  b_1, b_2 \in \gf_q }} 2 d_{b_1} d_{b_2} = \sum_{\substack{b_1 \ne b_2 \\ b_1, b_2\in \gf_q}} (d_{b_1} - d_{b_2})^2 = q - 1.
  \end{equation}
  
  Since the summation of $(d_{b_1} - d_{b_2})^2 $ is not $0$, there exist at least two distinct elements $b_1, b_2 \in \gf_q$ such that $|d_{b_1} - d_{b_2}| \geq 1$. Without loss of generality, suppose that there are $k$ elements $b \in \gf_q$ with $d_b = d_{b_1}$, where $1 \leq k \leq q-1$. Then there are $q - k$ elements $b \in \gf_q$ left with $d_b \ne d_{b_1}$ (these $q - k$ of $d_b$ may or may not equal to $d_{b_2}$). Thus, we have   
  $$
  q - 1 = \sum_{\substack{b_1 \ne b_2 \\ b_1, b_2 \in \gf_q}} (d_{b_1} - d_{b_2})^2 \geq k (q - k),
  $$
  and further we get $k \ge q -1$ or $k \le 1$. It then follows that (\ref{eqn-ban3rd}) is satisfied if and only if $d_c = \hat{d}$ for some $c \in \gf_q$ and $d_b = \hat{d} + 1$ or $d_b = \hat{d} - 1$ for each $b \ne c$. Note that $\sum_{b \in \gf_q} d_b = q^n - 1$, hence we have $d_c = q^{n-1} - 1$ for some $c \in \gf_q$ and $d_b = q^{n-1}$ for each $b \ne c$. The proof is then completed.
  
 \end{proof}

  

\begin{remark}\label{rmk-balance}
  By Theorem~\ref{thm-balance}, we see that every difference balanced function must be balanced, or an affine shift of a balanced function. In the sequel, without loss of generality, we may always assume that there is one zero less in the image set of a difference balanced function $f$ (otherwise, let $f = f - b$ for a suitable $b \in \gf_q^*$). 
  
\end{remark}

\section{$d$-homogeneous difference balanced functions}\label{sec-equi}

In this section, we give several necessary and sufficient conditions for $d$-homogeneous difference balanced functions from $\gf_{q^n}^*$ to $\gf_q$.

\begin{theorem}\label{thm-equi}
For a function $f$ from $\gf_{q^n}^*$ onto $\gf_q$, the following four statements are pairwise equivalent: 
  \begin{itemize}
    \item[(i)] $f$ is $d$-homogeneous and difference balanced;
    \item[(ii)] $f$ is two-tuple balanced;
    \item[(iii)] $D := \{(x, f(x)): x \in \gf_{q^n}^* \} \subseteq G = \gf_{q^n}^* \times \gf_q$ is a generalized difference set with parameters $(q(q^n-1); q, q^n - 1; q^n - 1, q^{n-1}; 0, q^{n-1} - 1)$ relative to $(\gf_q,+)$ and $(\gf_{q^n}^*, \cdot)$, and $(1, t)$ with $t \in \gf_q^*$ are multipliers of $D$ with $D^{(1,t)} = (a,0)D$ for some $a \in \gf_{q^n}^*$; 
    \item[(iv)] there exists a $\left( \frac{q^n-1}{q-1}, q-1, q^{n-1}, q^{n-2} \right)$ difference set in $\gf_{q^n}^* $ relative to $\gf_q^*$.
  \end{itemize}
\end{theorem}

\begin{proof}
  
  \begin{itemize}
    \item (i) $\Rightarrow$ (ii): See Theorem~\ref{thm-ttb-db}.
   
    \item (ii) $\Rightarrow$ (i): Since the two-tuple balance property implies the difference balance property, it suffices to prove that $f$ is $d$-homogeneous. By the definition of the two-tuple balance property, for each $a \in \gf_q^*$, there exists some $\mu_a \in \gf_q^*$ such that $f(a x) = \mu_a f(x)$ for all $x \in \gf_{q^n}^*$. It follows that $\mu_a = 1$ if and only if $a = 1$, and further $a = a'$ if and only if $\mu_a = \mu_{a'}$. Define a mapping
      $\epsilon: a \mapsto \mu_a$, we then have $\epsilon( a a') = \mu_a \mu_{a'} = \epsilon(a) \epsilon(a')$. Thus, $\epsilon$ is indeed an automorphism of $\gf_q^*$. This means there must exist a $d$ with $\gcd(d, q-1) = 1$ such that $\mu_a = a^d$, and then $f$ is $d$-homogeneous. 

    \item (i) $\Rightarrow$ (iii): Since $f$ is difference balanced from $\gf_{q^n}^*$ to $\gf_q$, we see that $D := \{ (x, f(x)): x \in \gf_{q^n}^* \}$ is a generalized difference set. Note that $f$ is $d$-homogeneous by assumption. For each $t \in \gf_q^*$, there exists a unique $a \in \gf_q^*$ such that $a^d = t$. Thus, we have
      \begin{eqnarray*}
	D^{(1,t)} & = & \{ (x, tf(x)): x \in \gf_{q^n}^* \} \\
	& = & \{ (x, f(ax)): x \in \gf_{q^n}^* \} \\
	& = & \{ (a^{-1} x, f(x) ): x \in \gf_{q^n}^* \} \\
	& = & (a^{-1},0)D,
      \end{eqnarray*}
  which means that $(1,t)$ with $t \in \gf_q^*$ are multipliers of $D$ with $D^{(1,t)} = (a,0) D $ for some $a \in \gf_{q^n}^*$. 
  
\item (iii) $\Rightarrow$ (i): It suffices to prove that $f$ is $d$-homogeneous. Since $D^{(1,t)} = (a_t, 0)D$, for each $x \in \gf_{q^n}^*$ there exists a unique $y$ with $x = a_t y$ and $t f(x) = f(y)$, which means $f(a_t^{-1} x) = t f(x)$ for all $x \in \gf_{q^n}^*$. It then follows that $t = 1$ if and only if $a_t=1$, and further $t = t'$ if and only if $a_t = a_{t'}$. Since $t \in \gf_q^*$, for each $a_t$, we have 
     $$
     f(\underbrace{a_t^{-1} \cdots a_t^{-1}}_{p-1} x) = t^{p-1} x = x.
     $$
     Then $a_t^{p-1} = 1$ and each $a_t \in \gf_q^*$. Thus, the mapping $\epsilon: t \mapsto a_t$ defines an automorphism of $\gf_q^*$: $\epsilon (t  t') = a_t a_{t'} =  \epsilon(t) \epsilon(t')$. Hence, there must be a $d$ with $\gcd(d, q-1) = 1$ such that $t = a_t^d$, which means $f$ is indeed $d$-homogeneous.
       
   \item (ii) $\Rightarrow$ (iv): Define
  $$
  D_b := \{ x \in \gf_{q^n}^*: f(x) = b \},
  $$
  where $b \in \gf_q^*$. Then by the definition of the two-tuple balance property, we have
  \begin{eqnarray*}
    D_b D_b^{(-1)} & = & q^{n-1} + q^{n-2}  \left( \gf_{q^n}^* -  \gf_q^* \right) \\
    & = & q^{n-1} + q^{n-2} \gf_{q^n}^*  - q^{n-2}  \gf_q^* ,
  \end{eqnarray*}
  which means that $D_b$ is a difference set in $\gf_{q^n}^*$ relative to $\gf_q^*$ with parameters $\left( \frac{q^n-1}{q-1}, q-1, q^{n-1}, q^{n-2} \right)$.
  
  \item (iv) $\Rightarrow$ (ii): Suppose that $C \subseteq \gf_{q^n}^*$ is a relative difference set with the parameters described in (iv). Define $C_b$ as
  $$
  C_b := \{ bx : x \in C \},
  $$
  for all $b \in \gf_q^*$, and $C_0 := \gf_{q^n}^* \setminus \gf_q^* C $. Then $\gf_{q^n}^*$ is the disjoint union of all the $C_b$'s for $b \in \gf_q$. We then define a function $f: \gf_{q^n}^* \rightarrow \gf_q$ as 
  $$
  f(x) := b^d, \textrm{ if $x \in C_b$, where $\gcd(d,q-1) = 1$}.
  $$
  We now show that $f$ is $d$-homogeneous and difference balanced. Since $C$ is a $( (q^n-1)/(q-1), q-1, q^{n-1}, q^{n-2} )$ difference set relative to $\gf_q^*$, we have the group ring equation:
  $$
  C C^{(-1)} = q^{n-1} + q^{n-2} \gf_{q^n}^* - q^{n-2} \gf_q^*.
  $$
It then follows that
\begin{eqnarray}\label{eqn-equi31}
  C_0 C_0^{(-1)} & = &  (\gf_{q^n}^* - \gf_q^* C) (\gf_{q^n}^* - \gf_q^* C^{(-1)}) \nonumber \\
  & = & \gf_{q^n}^* \gf_{q^n}^*  -  \gf_q^* \gf_{q^n}^* C^{(-1)}  -  \gf_q^* \gf_{q^n}^* C + \gf_q^* \gf_q^* CC^{(-1)} \nonumber \\
  & = & (q^n - 1) \gf_{q^n}^*  - 2(q - 1) q^{n-1} \gf_{q^n}^*  + (q-1) \gf_q^* C C^{(-1)} \nonumber \\
  & = & (-q^n + 2 q^{n-1} - 1 ) \gf_{q^n}^* \nonumber \\
   &  &  + (q-1) \gf_q^* (q^{n-1} + q^{n-2} \gf_{q^n}^* - q^{n-2} \gf_q^*) \nonumber \\
  & = & (q^{n-1} - 1 ) + (q^{n-1} - 1)  (\gf_q^* - \{1\} ) \nonumber \\
  &   & + (q^{n-2} - 1)  (\gf_{q^n}^* - \gf_q^*). 
\end{eqnarray}
Similarly, for $b \in \gf_q^*$, we have
\begin{equation}\label{eqn-equi32}
  C_b C_b^{(-1)} =  q^{n-1} + q^{n-2} (\gf_{q^n}^* - \gf_q^* );
\end{equation}
for $b \in \gf_q^*$, we have
\begin{equation}\label{eqn-equi33}
  C_0 C_b^{(-1)} = q^{n-2} \left( (\gf_{q^n}^* - \gf_q^*) \times \{- b\} \right)
\end{equation}
and
\begin{equation}\label{eqn-equi34}
C_b C_0^{(-1)}  = q^{n-2} \left( (\gf_{q^n}^* - \gf_q^*) \times \{ b\} \right) .
\end{equation}
For $b_1, b_2 \in \gf_q^*$ with $b_1 \ne b_2$, we have
\begin{eqnarray}\label{eqn-equi35}
  \lefteqn{C_{b_1} C_{b_2}^{(-1)}} \\
  & = & q^{n-1} (b_1b_2^{-1}, b_1 - b_2) + q^{n-2} ( (\gf_{q^n}^* - \gf_q^*) \times  \{b_1 - b_2\} ). \nonumber
\end{eqnarray}
In fact, (\ref{eqn-equi31})-(\ref{eqn-equi35}) are equivalent to the two-tuple balance property of the function $f$. For instance, by (\ref{eqn-equi35}), we have
$$
N_{b_1, b_2}(a) = \left\{ \begin{array}{cl}
  q^{n-1} & \textrm{ if $a^d = b_1 b_2^{-1}$,} \\
  q^{n-2} & \textrm{ if $a \not\in \gf_q^*$,}\\
  0 & \textrm{ otherwise.}
\end{array}\right.
$$
Similarly, one can verify that $f$ is two-tuple balanced by (\ref{eqn-equi31})-(\ref{eqn-equi35}). The proof is then completed.  

  \end{itemize}

\end{proof}

\begin{remark}\label{rmk-cds}
  By (i) and (iv) in Theorem~\ref{thm-equi}, if $f: \gf_{q^n}^* \rightarrow \gf_q$ is $d$-homogeneous and difference balanced, we see that the preimage set of $f$
  $$
  D_b := \{x \in \gf_{q^n}^*: f(x) = b\}
  $$
  is a $\left( \frac{q^n-1}{q-1}, q-1, q^{n-1}, q^{n-2} \right)$ difference set relative to $\gf_q^*$ for each $b \in \gf_q^*$, and
  $$
  D_0 := \{x \in \gf_{q^n}^*: f(x) = 0 \}
  $$
  is actually a divisible difference set with parameters 
  $$
  \left( \frac{q^n-1}{q-1}, q-1, q^{n-1} - 1, q^{n-1} - 1, q^{n-2}- 1 \right).
  $$
  By the canonical epimorphism $\rho: \gf_{q^n}^* \rightarrow \gf_{q^n}^*/\gf_q^*$, for each $b \in \gf_q^*$, the projection $\rho(D_b)$ is a cyclic difference set with Singer parameters 
  $$
  \left( \frac{q^n-1}{q-1}, q^{n-1}, q^{n-2}(q-1) \right),
  $$
  and $\rho(D_0)$ is the complement of $\rho(D_b)$, i.e., a cyclic difference set with Singer parameters
  $$
  \left( \frac{q^n-1}{q-1}, \frac{q^{n-1}-1}{q-1}, \frac{q^{n-2} -1}{q-1} \right).
  $$
  This unifies some of the results in~\cite{CX03,No04,KNCH05}.
\end{remark}

Because of the equivalent relation between (i) and (iv) in Theorem~\ref{thm-equi}, the Gordon-Mills-Welch construction~\cite{GMW62} and the No construction~\cite{No04} of difference balanced functions from $\gf_{q^n}^*$ to $\gf_q$ can be reformulated as the product of two relative difference sets as follows.

\begin{corollary}\cite[Corollary 3.2.2]{Pott95}\label{coro-casc}
  Suppose that $D_1$ is a difference set in $\gf_{q^n}^*$ relative to $\gf_{q^\ell}^*$ with parameters $\left( \frac{q^n - 1}{q^\ell - 1}, q^\ell - 1, q^{n-\ell}, q^{n-2\ell} \right)$ , where $\ell$ is a divisor of $n$, and $D_2$ is a $\left( \frac{q^\ell - 1}{q-1}, q - 1, q^{\ell-1}, q^{\ell-2} \right)$ difference set in $\gf_{q^\ell}^*$ relative to $\gf_q^*$. Then the set 
  $$
  D := \{ d_1 d_2: d_1 \in D_1, d_2 \in D_2 \}
  $$
  is a difference set in $\gf_{q^n}^*$ relative to $\gf_q^*$ with parameters
  $$
  \left( \frac{q^n-1}{q-1}, q - 1, q^{n-1}, q^{n-2} \right).
  $$
\end{corollary}

\section{The multipliers of generalized difference sets from difference balanced functions}\label{sec-multiplier}

Recalling the equivalent relation between (i) and (iii) of Theorem~\ref{thm-equi}, the $d$-homogeneity of functions from $\gf_{p^n}^*$ to $\gf_p$ is associated with certain numerical multipliers of the corresponding generalized difference sets.

\begin{corollary}\label{coro-multi-dhom}
  Suppose that $D := \{(x, f(x)): x \in \gf_{p^n}^*\} \subseteq G = (\gf_{p^n}^*, \cdot) \times (\gf_p, +)$. Then $(1,t)$, where $t = 1, \ldots, p-1$, are numerical multipliers of $D$ with $D^{(1, t)} = (a,0) D$ for some $a \in \gf_{p^n}^*$ if and only if $f$ is a $d$-homogeneous function from $\gf_{p^n}^*$ to $\gf_p$. 
\end{corollary}

In this section, we first determine the numerical multipliers of generalized $(p(p^n - 1); p, p^n -1; p^n - 1, p^{n-1}; 0, p^{n-1}- 1)$ difference sets in $G = (\gf_{p^n}^*, \cdot) \times (\gf_p,+)$ where $p$ is a prime. As an application of the multipliers, we are able to show that every difference balanced function from $\gf_{p^n}^*$ to $\gf_p$ must be $d$-homogeneous, which thus proves the Gong-Song conjecture for $q$ prime. 

To get information about $\chi(D)$, the decomposition of a prime $p$ into prime ideals in the ring $\bbZ[\zeta_w]$ is useful and actually explicitly known.

\begin{fact}\label{fact-decomp}\cite{Lan83}
 Let $p$ be a prime and $\zeta_w$ be a primitive $w$-th root of unity in $\bC$.  
  \begin{itemize}
    \item[(i)] If $w = p^e$, then the decomposition of the ideal $(p)$ in $\bQ(\zeta_w)$ into prime ideals is $(p) = (1 - \zeta_w)^{\phi(w)}$, where $\phi$ denotes the Euler's totient function. 
    \item[(ii)] If $(w,p) = 1$, then the prime ideal decomposition of $(p)$ is $(p) = \pi_1 \ldots \pi_v$, where the $\pi_i$'s are distinct prime ideals. Furthermore, $v = \phi(w)/o$ where $o$ is the order of $p$ modulo $w$. The field automorphism $\zeta_w \mapsto \zeta_w^p$ fixes the ideals $\pi_i$.
    \item[(iii)] If $w = p^e w'$ with $(w',p) = 1$, then the prime ideal $(p)$ decomposes as $(p) = (\pi_1 \ldots \pi_v)^{\phi(p^e)}$, where the $\pi_i$'s are distinct prime ideals and $v = \phi(w')/o$. Here $o$ is the order of $p$ modulo $w'$. If $t$ is an integer not divisible by $p$ and $t \equiv p^s \pmod{w'}$ for a suitable integer $s$, then the field automorphism $\zeta_w \mapsto \zeta_w^t$ fixes the ideals $\pi_i$. 
  \end{itemize}
\end{fact}

In our proof later, we also need the following lemma.

\begin{lemma}\label{lem-ma}~\cite{Ma85,Pott95}
  Let $G \cong H \times E$ be a finite abelian group with a cyclic Sylow $p$-subgroup of order $p^s$. Let $\chi$ be a character of order $p^s$. If $\gamma\chi(A) \equiv 0 \pmod{p^r}$ for all characters $\gamma$ of $H$, then we have
  $$
  A = p^r X + E' Y,
  $$
  where $A, X, Y \in \bbZ[G]$ and $E'$ is the unique subgroup of $G$ of order $p$, i.e., the coefficients in $A$ are congruent modulo $p^r$ on cosets of $E'$.
\end{lemma}

Now we determine the numerical multipliers of generalized difference sets from difference balanced functions. We remark that the proof idea comes from the proofs of typical multiplier theorems (see for instance~\cite{BJL99}).

\begin{theorem}\label{thm-multiplier}
  Let $D$ be a generalized difference set with parameters $(p(p^n-1); p, p^n - 1; p^n - 1, p^{n-1}; 0, p^{n-1}- 1)$ in the group $G = N_2 \times N_1$, where $N_2 = (\gf_{p^n}^*, \cdot)$, $N_1 = (\gf_p,+)$ and $p$ is a prime. Then $(1,t_i)$ are numerical multipliers of $D$, where $t_i$ are the integers
  $$
  t_i := p + i(p^n - 1), \quad i = 1, \ldots, p-1.  
  $$
 \end{theorem}

\begin{proof}
  Let $w = (p^n - 1)p$ and $t \equiv p \pmod{p^n - 1}$, then by (iii) of Fact~\ref{fact-decomp} and (\ref{eqn-char}), we have
  $$
  \chi(D^{(1,t)}) \chi(D^{(-1)}) \equiv \chi(D) \chi(D^{(-1)}) \equiv 0 \pmod{p^n}
  $$
  for characters $\chi|N_1 \ne \chi_0$ and $\chi|N_2 \ne \chi_0$. Since $DN_1 = D^{(1, t)} N_1 = G$, if $\chi|N_1 = \chi_0$ and $\chi \ne \chi_0$, we get $\chi(D^{(1, t)}) = \chi(D) = 0$. On the other hand, note that $D^{(1, t)}N_2 = DN_2 = p^{n-1}G - N_2$ because $f$ is balanced by Theorem~\ref{thm-balance} and Remark~\ref{rmk-balance}. Then $\chi(D^{(1, t)}) = \chi(D) = - 1$ if $\chi|N_2 = \chi_0$ and $\chi \ne \chi_0$. Thus, by (\ref{eqn-char}), we have
  \begin{equation}\label{eqn-char2}
    \chi(D^{(1, t)}) \chi(D^{(-1)}) \equiv \left\{ \begin{array}{cl}
      0 \bmod{p^n} & \textrm{ if $\chi|N_1 \ne \chi_0$ and $\chi|N_2 \ne \chi_0$}, \\
      0 \bmod{p^n} & \textrm{ if $\chi|N_1 = \chi_0$ and $\chi \ne \chi_0$}, \\
      1 \bmod{p^n} & \textrm{ if $\chi|N_2 = \chi_0$ and $\chi \ne \chi_0$}, \\
      1 \bmod{p^n} & \textrm{ if $\chi = \chi_0$}.
    \end{array}\right.
  \end{equation}

  Suppose that $D^{(1, t)}  D^{(-1)} = p^{n-1} G - N_2 + F$, where $F$ is a suitable element in $\bbZ[G]$. Note that in $\bbZ[G]$ all the coefficients of $F$ are greater than or equal to $-p^{n-1}$. It then follows from (\ref{eqn-char2}) that   
  \begin{equation}\label{eqn-fmod}
\chi(F) \equiv 0 \pmod{p^n}
  \end{equation}
  for all characters $\chi$.

  Now, we calculate $D^{(1, t)}(D^{(-1)})^{(1,t)} D D^{(-1)}$ in two different ways: one by $(D^{(1, t)}(D^{(-1)})^{(1,t)}) (D D^{(-1)})$, and the other by $(D^{(1, t)}D^{(-1)}) ( (D^{(-1)})^{(1,t)} D)$. With (\ref{eqn-gds}) we then have
\begin{eqnarray}\label{eqn-fourd}
  \lefteqn{(p^n + p^{n-1} G - N_2 - p^{n-1} N_1) (p^n + p^{n-1} G - N_2 - p^{n-1} N_1)} \nonumber \\
  & = & (p^{n-1} G - N_2 + F) (p^{n-1} G - N_2 + F^{(-1)}). 
\end{eqnarray}
Since $D^{(1, t)} G = DG = (p^n - 1)G$, we have $( D^{(1, t)} G) D^{(-1)} =  ( D G) D^{(-1)}$. We compute the left-hand side as
$$
(D^{(1,t)} G) D^{(-1)} = (D^{(1,t)} D^{(-1)} ) G =  (p^{n-1}G - N_2 + F) G;
$$
and compute the right-hand side as
$$
( D G) D^{(-1)} = (D D^{(-1)}) G = (p^n + p^{n-1}G - N_2 - p^{n-1} N_1 ) G.
$$
Thus, we have
\begin{equation*}
  (p^{n-1}G - N_2 + F) G = (p^n + p^{n-1}G - N_2 - p^{n-1} N_1) G, 
\end{equation*}
which implies that
\begin{equation*}
  FG = p^n G - p^{n-1}N_1 G = p^{n-1} (pG - N_1 G) = 0.
\end{equation*}
By computing $D^{(1, t)}G (D^{(-1)})^{(1,t)} = DG (D^{(-1)})^{(1,t)}$ in the same way for both sides, we can also get 
\begin{equation}\label{eqn-fourd1}
  F^{(-1)} G = FG = 0.
\end{equation}
Recall that $D^{(1, t)} N_2 = DN_2 = p^{n-1} G - N_2$, then we have $(D^{(1, t)} N_2 ) D^{(-1)} = ( DN_2 ) D^{(-1)}$, and further $ ( D^{(1,t)} D^{(-1)}) N_2 = (DD^{(-1)}) N_2$ by changing the order on both sides. It then follows that
\begin{equation*}
  (p^{n-1}G - N_2 + F) N_2 = (p^n + p^{n-1}G - N_2 - p^{n-1} N_1) N_2 ,
\end{equation*}
and thus,
$$
FN_2 = p^n N_2 - p^{n-1} N_1 N_2.
$$
Similarly, by $(D^{(1,t)} N_2) (D^{(-1)})^{(1,t)} = (DN_2 ) (D^{(-1)})^{(1,t)}$ we get
\begin{equation}\label{eqn-fourd2}
  F^{(-1)}N_2  = FN_2 = p^n N_2 - p^{n-1} N_1 N_2.
\end{equation}
With (\ref{eqn-fourd1}) and (\ref{eqn-fourd2}), simplifying (\ref{eqn-fourd}), we have
\begin{equation}\label{eqn-ffinv}
  F F^{(-1)} = p^{2n} - p^{2n-1} N_1,
\end{equation}
where $F \in \bbZ[G]$ and $\chi(F) \equiv 0 \pmod{p^n}$ for all characters $\chi$ of $G$. 

Since $N_1 = (\gf_p, +)$ is a cyclic group of prime order $p$, the integers which are relatively prime to $p$ and congruent $p$ modulo $p^n - 1$ are precisely the $t_i$'s of the form $t_i = p + i(p^n - 1)$ with $i = 1, \ldots, p - 1$. Let $F = \sum_{g \in G} f_g g$, then using Lemma~\ref{lem-ma} and (\ref{eqn-fmod}), we know that all $f_g$'s are congruent modulo $p^n$ on cosets of $N_1$. Recall that $FF^{(-1)} = p^{2n} - p^{2n-1}N_1 $, then the coefficient of the identity in $FF^{(-1)}$ is 
\begin{equation}\label{eqn-ffcoe}
  \sum_{g \in G} f_g^2 = p^{2n} - p^{2n-1}, 
\end{equation}
and $FF^{(-1)}$ is not constant on the cosets of $N_1$. This means that there must be at least one coset $ g N_1$ with $g \in G$ such that $f_u \ne f_v$ for $u, v \in g N_1 $. Let $f_u \equiv y \pmod{p^n}$ for all $u \in g N_1 $, where $y$ is the smallest of the $f_u$'s. Then there is at least one $u \in g N_1 $ such that the coefficient $f_u \ge y + p^n$. Since by assumption $f_g \ge - p^{n-1}$ for all $g \in G$, we have $y \ge - p^{n-1}$ and $(y + p^n)^2 > y^2$. Thus, we have
\begin{eqnarray*}
  \lefteqn{\sum_{u \in g N_1} f_u^2} \\
  & \geq & (- p^{n-1} + \delta)^2 (p-1) + (- p^{n-1} + \delta + p^n)^2 \\
  & = & p^{2n} - p^{2n - 1} + \delta^2 p \\ 
  & \geq & p^{2n} - p^{2n - 1},
\end{eqnarray*}
where $\delta \geq 0$ and the equality holds if and only if $\delta = 0$. Because of (\ref{eqn-ffcoe}), we have $f_u = p^n - p^{n-1}$ for some $u \in g N_1 $ and $f_v = -p^{n-1}$ for all the other $v \in g N_1$ with $v \ne u$.

To sum up, we proved that $F = p^n u - p^{n-1} g N_1 $ is the only solution to (\ref{eqn-ffinv}), and we now know that $D^{(1,t)} D^{(-1)} = p^{n-1} G - N_2 + p^n u - p^{n-1} g N_1 $. Calculating $D^{(1, t)} D^{(-1)} D$ as $(D^{(1, t)} D^{(-1)} ) D = D^{(1, t)} ( D^{(-1)} D) $, we have
\begin{eqnarray*}
  \lefteqn{ (p^{n-1} G - N_2 + p^n u - p^{n-1} g N_1 ) D} \\
  & = & D^{(1, t)} ( p^n + p^{n-1} G - N_2 - p^{n-1} N_1 ),  
\end{eqnarray*}
and further $D^{(1, t)} = u D$ because $D^{(1,t)} N_1 = D N_1 = G$, $D^{(1,t)} N_2 = DN_2 = p^{n-1}G - N_2$ and $D^{(1,t)} G = DG = (p^n - 1) G$ as we stated earlier, which completes the proof. 
\end{proof}

 %
 %

As an application of the multipliers, we now can prove the conjecture by Gong and Song~\cite{GS06} for $q$ prime.

\begin{theorem}\label{thm-main}
  Suppose that a function $f : \gf_{p^n}^* \rightarrow \gf_p$ is difference balanced. Then $f$ must be $d$-homogeneous, i.e., $f(ax) = a^d f(x)$ for each $a \in \gf_p^*$ with $\gcd(d, p-1) = 1$. 
\end{theorem}

\begin{proof}
  The conclusion follows from Theorem~\ref{thm-multiplier} and Corollary~\ref{coro-multi-dhom}. 
\end{proof}

Because of the result above, we have a better understanding of difference balanced functions from $\gf_{p^n}^*$ to $\gf_p$, and some previous results in the literature are now more transparent. 

\begin{corollary}\label{coro-main}
  Suppose that $f: \gf_{p^n}^* \rightarrow \gf_p$ is difference balanced where $p$ is a prime, or equivalently, the sequence ${\bf s} = (s_i)$, where $s_i = f(\theta^i)$ and $\theta$ is a primitive element of $\gf_{p^n}$, has the ideal two-level autocorrelation, then the following hold:
  \begin{itemize}
    \item[--] $f$ is two-tuple balanced;
    \item[--] the preimage set 
      $$
      D_b := \{ x \in \gf_{p^n}^*: f(x) = b\}
      $$
      is a cyclic $\left(\frac{p^n-1}{p-1}, p-1, p^{n-1}, p^{n-2} \right)$ difference set relative to $\gf_p^*$ for each $b \in \gf_p^*$, and is a cyclic divisible difference set with parameters $\left( \frac{p^n-1}{p-1}, p-1, p^{n-1}-1, p^{n-1} -1, p^{n-2} - 1\right)$ for $b = 0$. Furthermore, the projection $\rho(D_b)$ is a cyclic difference set with Singer parameters $\left(\frac{p^n-1}{p-1}, p^{n-1}, p^{n-2}(p-1) \right)$ for each $b \in \gf_p^*$ and $\rho(D_0)$ is the complement
      of $\rho(D_b)$, where $\rho: \gf_{p^n}^* \rightarrow \gf_{p^n}^*/\gf_p^*$ is the canonical epimorphism.
  \end{itemize}
\end{corollary}

\section{Conclusions}\label{sec-con}

It seems that the method in Section~\ref{sec-multiplier} does not work when $q$ is a prime power since the structure of the automorphism group $\aut(\gf_q, +)$ is more complicated. By Theorem~\ref{thm-main}, difference balanced functions from $\gf_{p^n}^*$ to $\gf_p$ are now characterized by the $d$-homogeneity. It may be possible and quite interesting to classify all existing $p$-ary sequences with the ideal two-level autocorrelation, or find new inequivalent constructions of such $p$-ary sequences.

\section*{Acknowledgement}
Qi Wang's research was supported by the Alexander von Humboldt (AvH) Stiftung/Foundation.



\providecommand{\bysame}{\leavevmode\hbox to3em{\hrulefill}\thinspace}
\providecommand{\MR}{\relax\ifhmode\unskip\space\fi MR }
\providecommand{\MRhref}[2]{%
  \href{http://www.ams.org/mathscinet-getitem?mr=#1}{#2}
}
\providecommand{\href}[2]{#2}

\end{document}